\documentclass[12pt]{amsart}
\advance\hoffset by -0.5 truecm
\usepackage{amssymb}
\newtheorem{Theorem}{Theorem}[section]
\newtheorem{Lemma}[Theorem]{Lemma}

\newtheorem{Proposition}[Theorem]{Proposition}

\newtheorem{Definition}[Theorem]{Definition}
\newtheorem{Remark}[Theorem]{Remark}

\def\V{\mbox{Var}}

\def\Z{{\mathbb Z}}
\def\R\re
\def\V{\bf V}

\def \la{\lambda}

\def \re{{\mathbb R}}

\def \C{{\mathbb C}}
\def \M{{\widetilde M}}
\def \I{{\mathbb I}}
\def \H{{\mathbb H}_{\bf C}}

\def \0{\lambda_{0}}
\def \la{\lambda}
\def \ga{\gamma}

\def \G{{\mathbf G}_{\mu}}
\def\P{{\mathbb P}}

\newcommand{\de}[2]{\frac{\partial #1}{\partial #2}}
\newcommand{\del}[2]{\frac{\delta #1}{\delta #2}}

\begin{document}

\title[]{On the cohomological equation of magnetic flows}

\author[N.S. Dairbekov]{Nurlan S. Dairbekov}
\address{Kazakh British Technical University,
Tole bi 59, 050000 Almaty, Kazakhstan }
\email{Nurlan.Dairbekov@gmail.com}

\author[G.P. Paternain]{Gabriel P. Paternain}
 \address{ Department of Pure Mathematics and Mathematical Statistics,
University of Cambridge,
Cambridge CB3 0WB, England}
 \email {g.p.paternain@dpmms.cam.ac.uk}

\begin{abstract} We consider a magnetic flow without conjugate points on a closed manifold $M$ with generating
vector field $\G$. Let $h\in C^{\infty}(M)$ and let $\theta$ be a smooth 1-form on $M$.
We show that the cohomological equation
\[\G(u)=h\circ \pi+\theta\]
has a solution $u\in C^{\infty}(SM)$ only if $h=0$ and $\theta$
is closed. This result was proved in \cite{DP2} under the assumption that the flow of $\G$ is Anosov.
\end{abstract}



\maketitle

\section{Introduction} In the present paper, which is a sequel to \cite{DP2},
 we study the cohomological equation for a certain
class of second order differential equations on the tangent bundle of a closed connected manifold
$M$ with canonical projection $\pi:TM\to M$ ($\pi(x,y)=x$ for $x\in M$, $y\in T_xM$). 

The dynamical systems that we will consider are magnetic flows. In their most general form they are
determined by a pair $(F,\Omega)$, where $F$ is a Finsler metric on $M$ and $\Omega$ is a closed 2-form.
The Legendre transform $\ell_{F}:TM\setminus\{0\}\to T^*M\setminus\{0\}$
associated with the Lagrangian $\frac{1}{2}F^2$ is a diffeomorphism and
$\omega_{0}:=\ell_{F}^{*}(-d\la)$ defines a symplectic form on $TM\setminus\{0\}$, where $\la$
is the Liouville 1-form on $T^*M$. 
The {\it magnetic flow} of the pair $(F,\Omega)$ is the Hamiltonian flow $\phi$ of
$\frac{1}{2}F^2$ with respect to the symplectic form $\omega_{0}+\pi^*\Omega$.

Contrary to what happens for geodesic flows, a magnetic flow can change behaviour drastically as we change
energy levels. One can obtain the behaviour in all energy levels by restricting the flow
$\phi$ to the unit sphere bundle $SM:=F^{-1}(1)$ and changing $\Omega$ by $\la\,\Omega$, where $\la\in\re$.
Henceforth, we shall only consider $\phi$ acting on $SM$ and we will denote by $\G$ the infinitesimal generator
of $\phi$.
A curve $\ga:\re\to M$ given by $\ga(t)=\pi(\phi_{t}(x,y))$ will be called a {\it magnetic geodesic}.

The {\it cohomological equation} (also known in the literature as the {\it kinetic equation}) is simply
\[\G(u)=v\]
where $u,v$ are functions on $SM$. The function $v$ is said to be a {\it coboundary}.
The importance of the cohomological equation in dynamical systems is well known; it arises for example
in the study of invariant measures, conjugacy problems, reparametrizations,
rigidity questions and inverse problems.

Suppose now that $v\in C^{\infty}(SM)$ and we try to solve $\G(u)=v$ for $u\in C^{\infty}(SM)$. 
There are some obvious obstructions for doing so. Let $\mathcal I(\phi)$ be the space of all invariant distributions,
i.e., $\mathcal D\in\mathcal I(\phi)$ is an element of the dual space of $C^{\infty}(SM)$ such that
$\mathcal D(\G(u))=0$ for every smooth function $u$. Hence, if $\G(u)=v$ admits a smooth solution
$u$ we must have $\mathcal D(v)=0$ for all $\mathcal D\in \mathcal I(\phi)$. In particular,
\[\int_{SM}v\,d\mu=0\]
for every $\phi$-invariant Borel probability measure $\mu$.
Remarkably, in some cases invariant distributions are the only obstructions.
For Anosov flows, the smooth version of the Liv\v sic theorem
(\cite{LMM}) says that there exists $u\in C^{\infty}(SM)$ such that $\G(u)=v$ where $v\in C^{\infty}(SM)$
if and only if $v$ has zero integral along every closed orbit of the flow. 
More recently, L. Flaminio and G. Forni \cite{FF} proved the following:

\begin{Theorem} Let $M$ be a closed oriented hyperbolic surface
and let $\Omega$ be the area form. If $\mathcal D(v)=0$ for all $\mathcal D\in\mathcal I(\phi)$
then there exists $u\in C^{\infty}(SM)$ such that
\[\G(u)=v.\]\label{flafor}
\end{Theorem}

\medskip

Flaminio and Forni proved this theorem for the usual {\it horocycle flow} $\phi^h$, but
it is easy to see that $\phi$ and $\phi^h$ are conjugated by the map
$(x,y)\mapsto (x,iy)$, where $iy$ is the vector obtained by rotating $y$ by $\pi/2$
following the orientation of the surface.
Using representation theory they were able to determine completely
the space $\mathcal I(\phi^h)$ in terms of the spectrum of the Laplacian of $M$ and
the genus of $M$. We will return to their results below.

In the present paper we are concerned with the following geometrical aspect of the
cohomological equation related to the fibration $SM\to M$.
Assume there exists $u,v\in C^{\infty}(SM)$ such that $\G(u)=v$ and suppose that
for every $x\in M$, $v(x,y)$ is a polynomial of degree $k$ in $y$. What does it imply about $u$?
Must $u$ also be a polynomial in $y$? Of special interest to us is the case $k=1$,
i.e., we suppose that $v=h\circ\pi+\theta$, where $h\in C^{\infty}(M)$ and
$\theta$ is a smooth 1-form which we view as a function $\theta:TM\to\re$.
The case $k=2$ is also of great interest, but with the exception of Proposition \ref{higher}
we will not consider it here.
Even for $k=1$ cohomological equations of this special kind
appear in problems related to spectral rigidity and regularity of the Anosov splitting
\cite{DP1,DP2}.
Note that if $h=0$ and $\theta=df$ is exact, then $\G(f\circ\pi)=\theta$
because $d\pi(\G(x,y))=y$.

The literature on this topic is abundant, see \cite{GK1,GK2,min,CS,DS,P,DP1,DP2}.
However in all these references the results require either non-positive
curvature (and transitivity) or uniform hyperbolicity of $\phi$.
Here we show that these conditions can be weakened to just no conjugate points.

\bigskip

\noindent {\bf Theorem A.}  {\it Suppose $SM$ does not have conjugate points.
Let $h\in C^{\infty}(M)$ and let $\theta$ be a smooth 1-form on $M$.
Then the cohomological equation
\[\G(u)=h\circ \pi+\theta\]
has a solution $u\in C^{\infty}(SM)$ only if $h=0$ and $\theta$
is closed.}

\bigskip

Given $(x,y)\in TM$ we define {\it the vertical subspace
at} $(x,y)$ as ${\mathcal V}(x,y):=\mbox{\rm ker}\,d_{(x,y)}\pi$, where $\pi:TM\to M$
is the canonical projection. 
We say that the orbit of $(x,y)\in SM$ {\it does not have conjugate
points} if for all $t\neq 0$,
$$d_{(x,y)}\phi_{t}({\mathcal V}(x,y))\cap {\mathcal V}(\phi_{t}(x,y))=\{0\}.$$
The energy level $SM$ is said to have no conjugate points if
for all $(x,y)\in SM$, the orbit of $(x,y)$ does not have conjugate points.
Since magnetic flows are optical, the main result in \cite{CGIP} says that
$SM$ has no conjugate points if and only if the asymptotic Maslov index 
of the Liouville measure ${\mathfrak m}(\mu)=0$.

Note that if $\phi$ is Anosov, then there are no conjugate points \cite{PP,P3,CGIP}
and Theorem A was proved in \cite{DP2} using a non-negative
version of the Liv\v sic theorem \cite{LT,PS} and an integral version of the Pestov identity.
Here we will use the same integral Pestov identity and show how we can
do without hyperbolicity.

Theorem A seems to be new even for the particular case
of the geodesic flow of a Riemannian metric ($\Omega=0$ and $F$ is Riemannian).
There are several interesting examples of geodesic flows without conjugate which
are not Anosov and have regions of positive curvature.
The example of W. Ballmann, M. Brin and K. Burns in \cite{BBB}
is of this kind and has the additional feature of having non-continuous 
Green subbundles (see Section 2). Another interesting class of magnetic flows
without conjugate points is given by compact quotients $M$ of complex hyperbolic space
$\H^n$. If we let $(g,\Omega)$ be the K\"ahler structure with holomorphic sectional curvature
$-1$ on $M$, then the magnetic flow $\phi$ of the pair $(g,\Omega)$ has no conjugate points
and $\phi$ is an algebraic unipotent flow. For $n=1$ we obtain the flow in
Theorem \ref{flafor}. The proof that these magnetic flows have no conjugate points
is fairly simple and is given in the appendix, where we also collect other facts
about them which are probably well known to experts, but not readily available
in the literature.

It is very likely that in Theorem A one can replace ``closed" by ``exact", but we do not
know how to prove this in general. Let us explain what are the complications and at the same
time indicate very general conditions under which we can claim that $\theta$ must be exact.

Let $\mathcal M(\phi)$ be the space of all $\phi$-invariant Borel probability measures.
Clearly $\mathcal M(\phi)\subset \mathcal I(\phi)$. To any element 
$\mathcal D\in\mathcal I(\phi)$ we can associate its {\it asymptotic cycle}
$\rho(\mathcal D)\in H_1(M,\re)$ by setting
\[\langle \rho(\mathcal D),[\omega]\rangle=\mathcal D(\omega),\]
where $[\omega]\in H^1(M,\re)$ and $\omega$ is any closed 1-form in the class
$[\omega]$. Invariance of $\mathcal D$ ensures that $\rho(\mathcal D)$ is well
defined.

Suppose now that $\rho: \mathcal I(\phi)\to H_1(M,\re)$ is surjective. Then it
is immediate to see that if $\G(u)=\theta$ with $\theta$ closed, then $\theta$
must in fact be exact. We know of no example of a magnetic flow without conjugate
points for which $\rho$ is not surjective. We indicate now several conditions that
imply the surjectivity of $\rho$.

If $\Omega=0$, then $\rho$ is surjective. This follows from the fact
that $\rho(\mathcal M(\phi))$ is a compact convex set containing the origin
in its interior, which in turn follows from the fact that every non-trivial homotopy
class contains a closed geodesic, see \cite[Chapter 1]{P1}.

If $\phi$ is Anosov, the closed orbits are dense and it is not hard to see
that $\rho$ is also surjective (cf. \cite{Pla}).

But one can say more. Let $\M$ be the universal covering
of $M$. Consider the exponential map $\exp_{x}:T_{x}\M\to\M$ of the energy level $S\M$, given by
$\exp_{x}(ty)=\pi\circ\phi_{t}(x,y)$, where $x\in\M$, $t\geq 0$ and $y\in S_{x}\M$.
It is unknown if the abscence of conjugate points in $SM$ implies that $\exp_x$
is a diffeomorphism for all $x$ (see \cite[p. 907]{CI}). But suppose it does and to simplify
matters suppose also that $F$ is a Riemannian metric.
Then, $\M$ is diffeomorphic to $\re^n$ and the lift $\widetilde\Omega$ of $\Omega$ to $\M$ is exact.
Write $\widetilde\Omega=d\vartheta$.
Then we can associate to the magnetic system a critical value $c$ in the sense of Ma\~n\'e \cite{BP,Plast}:
\[c= \inf_{f\in C^{\infty}(\M,\re)}\;\sup_{x\in \M}\;
   \frac{1}{2}|d_{x}f+\vartheta_{x}|^{2}.\]
(Note that as $f$ ranges over $C^{\infty}(\M,\re)$ the form $\vartheta+df$ ranges over all
primitives of $\widetilde\Omega$, because any two primitives differ by a closed 1-form
which must be exact since $\M$ is simply connected.)

We will say that a magnetic flow is {\it Ma\~n\'e critical} if $c=1/2$.
If $\exp_{q}$ is a diffeomorphism, the same proof of Theorem D in \cite{CIPP}
shows that $c\leq 1/2$. If $c<1/2$, then assuming a technical condition on $\pi_1(M)$ which seems
to hold always, it is known that every non-trivial homotopy class contains a closed magnetic geodesic
\cite{Plast} and again $\rho$ is surjective.

We are left with the question: suppose $SM$ has no conjugate points and is Ma\~n\'e critical,
is $\rho$ surjective?  We mentioned before that compact quotients of complex hyperbolic space
$\H^n$ with the K\"ahler structure with holomorphic sectional curvature
$-1$ have no conjugate points. It turns out that they are also Ma\~n\'e critical
(for $n=1$ this is proved in \cite[Example  6.2]{Co} and for $n\geq 2$ the proof is similar). 
We do not know other examples of Ma\~n\'e critical magnetic flows without conjugate points.

For $n=1$ we will check using Flaminio and Forni's explicit computation of the invariant distributions
that $\rho$ is surjective (see Section \ref{dt}). 
We suspect that the same is true for any $n\geq 2$, but finding the invariant
distributions and understanding the cohomological equation using representation theory and harmonic analysis
seems quite a laborious task. A direct argument using the Pestov integral identity obtained
in \cite{DP2} will allow us to show exactness in this case. More precisely we will show:

\medskip

\noindent {\bf Theorem B.} {\it Let $M$ be a smooth compact quotient of complex hyperbolic space
$\H^n$. Let $(g,\Omega)$ be the K\"ahler structure with holomorphic sectional curvature
$-1$ and let $\G$ be the vector field generated by the magnetic flow of the pair $(g,\Omega)$.
Let $h\in C^{\infty}(M)$ and let $\theta$ be a smooth 1-form on $M$.
Then the cohomological equation
\[\G(u)=h\circ \pi+\theta\]
has a solution $u\in C^{\infty}(SM)$ if and only if $h=0$ and $\theta$
is exact.

}

\medskip

We note that if Theorem \ref{flafor} extends to compact quotients of $\H^n$ for $n\geq 2$, then
Theorem B shows that $\rho$ must be surjective for $n\geq 2$.


\subsection{Convex Hamiltonians} Theorem A above can be used to prove a fairly general result
for an arbitrary {\it convex superlinear} Hamiltonian $H:T^*M\to\re$ ($M$ close and connected). 
The Hamiltonian is said to be convex if $\partial^2 H/\partial p^2$ is everywhere positive definite.
The Hamiltonian is superlinear if for all $x\in M$,
\[\lim_{|p|\to\infty}\frac{H(x,p)}{|p|}=+\infty.\]

Let $\tau:T^*M\to M$ be the canonical projection.
Given an arbitrary smooth closed 2-form $\Omega$ on $M$, we consider $T^*M$ endowed with the
symplectic structure $-d\lambda+\tau^*\Omega $
where $\la$ is the Liouville 1-form. Given a convex superlinear Hamiltonian $H:T^*M\to\re$
we let $X_{H}$ be the Hamiltonian vector field of $H$ with respect to $-d\lambda+\tau^*\Omega $.
We denote by $\phi$ the flow of $X_H$.
Let $c$ be a regular value of $H$ and set $\Sigma:=H^{-1}(c)$.

As before we say that the orbit of $(x,p)\in \Sigma$ {\it does not have conjugate
points} if for all $t\neq 0$,
$$d_{(x,p)}\phi_{t}({\mathcal V}(x,p))\cap {\mathcal V}(\phi_{t}(x,p))=\{0\},$$
where now ${\mathcal V}(x,p)=\mbox{\rm ker}\,d_{(x,p)}\tau$.
The energy level $\Sigma$ is said to have no conjugate points if
for all $(x,p)\in \Sigma$, the orbit of $(x,p)$ does not have conjugate points.

\bigskip

\noindent {\bf Theorem C.}  {\it Suppose $\Sigma$ does not have conjugate points
and let $\theta$ be a smooth 1-form on $M$.
Then the cohomological equation
\[X_{H}(u)=\tau^*\theta(X_{H})\]
has a solution $u\in C^{\infty}(\Sigma)$ only if $\theta$
is closed.}

\bigskip

\section{Preliminaries}

\subsection{Green subbundles} If $SM$ has no conjugate points, one can construct the so
called {\it Green subbundles} \cite[Proposition A]{CI} given by:
\[E(x,y):=\lim_{t\to +\infty}d\phi_{-t}({\mathcal V}(\phi_{t}(x,y))),\]
\[F(x,y):=\lim_{t\to +\infty}d\phi_{t}({\mathcal V}(\phi_{-t}(x,y))).\] 
These subbundles are Lagrangian, they never intersect the vertical subspace
and, crucial for us, they are contained in $T(SM)$. Moreover, they vary measurably
with $(x,y)$ and they contain the vector field $\G$.

Assume now that $SM$ has no conjugate points and let $E$ be one of Green subbundles.
Using the splitting $$T_{(x,y)}TM=\mathcal H(x,y)\oplus \mathcal V(x,y),$$
where $\mathcal H(x,y)$ is the horizontal subspace,
we can represent $E(x,y)$ as the graph of a linear map $S(x,y):T_{x}M\to T_{x}M$.
The correspondence $(x,y)\mapsto S(x,y)$ is measurable and $\|S\|\in L^{\infty}(SM)$
 \cite[Proposition 1.7]{CI}.

\subsection{Semibasic tensor fields}
For the reader's convenience we recall
various definitions and notations from \cite{DP2}.
Henceforth $M$ is a closed $n$-dimensional manifold and
$F$ is a Finsler metric on~$M$.

Let $\pi:TM\setminus\{0\}\to M$ be the natural projection, and let
$\beta^r_s M:=\pi^*\tau^r_s M$ denote the bundle of semibasic
tensors of degree $(r,s)$, where $\tau^r_s M$ is the bundle of
tensors of degree $(r,s)$ over $M$. Sections of the bundles
$\beta^r_sM$ are called semibasic tensor fields and the space of
all smooth sections is denoted by $C^{\infty}(\beta^r_sM)$. For
such a field $T$, the coordinate representation
$$
T=(T^{i_1\dots i_r}_{j_1\dots j_s})(x,y)
$$
holds in the domain of a standard local coordinate system $(x^i,y^i)$
on $TM\setminus\{0\}$
associated with a local coordinate system $(x^i)$ in $M$.
Under a change of a local coordinate system, the components of
a semibasic tensor field are transformed by  the same formula
as those of an ordinary tensor field on $M$.

Every ``ordinary'' tensor field on $M$ defines
a semibasic tensor field by the rule $T\mapsto T\circ\pi$, so that
the space of tensor fields on $M$ can be treated as embedded in the space
of semibasic tensor fields.

Let $(g_{ij})$ be the fundamental tensor,
$$
g_{ij}(x,y)=\frac12[F^2]_{y^iy^j}(x,y),
$$
and let $(g^{ij})$ be the contravariant fundamental tensor,
\begin{equation}\label{g-1}
g_{ik}g^{kj}=\delta_i^j.
\end{equation}

In the usual way, the fundamental tensor
defines the inner product $\langle\cdot,\cdot\rangle$
on $\beta^1_0 M$, and we put $|U|^2=\langle U,U\rangle$.

Let
$$
\mathbf G =y^i\de{}{x^i}-2G^i\de{}{y^i}
$$
be the spray induced by $F$. Here
$G^i$ are the geodesic coefficients \cite[(5.7)]{She},
$$
G^i(x,y)=\frac14g^{il}
\left\{2\de{g_{jl}}{x^k}-\de{g_{jk}}{x^l}\right\}y^jy^k.
$$

Let
\begin{equation*}
T(TM\setminus\{0\})
=\mathcal H TM\oplus \mathcal V TM
\end{equation*}
be the decomposition of $T(TM\setminus\{0\})$
into horizontal and vertical vectors. Here
$$
\mathcal H TM=\operatorname{span}\left\{\del{}{x^i}\right\},
\quad
\mathcal V TM =\operatorname{span}\left\{\de{}{y^i}\right\},
$$
with
$$
\del{}{x^i}=\de{}{x^i}-N^j_i\de{}{y^j}
$$
and
$$
N^i_j=\de{G^i}{y^j}.
$$

Let
$$
\nabla:C^{\infty}(T(TM))\times C^\infty(\pi^*TM)\to
C^{\infty}(\pi^*TM)
$$
be the Chern connection,
$$
\nabla_{\hat X}U=\left\{dU^i(\hat X)+U^j\omega_j^i(\hat X)\right\}\de{}{x^i},
$$
where
$$
\omega^i_j=\Gamma^i_{jk}dx^k
$$
are the connection forms. Recall that
\begin{equation}\label{nij}
N^i_j=\Gamma^i_{jk}y^k.
\end{equation}

Given a function $u\in C^\infty(TM\setminus\{0\})$, we put
$$
u_{|k}:=\del u{x^k},
\quad u_{\cdot k}:=\de u{y^k}
$$
and, given a semibasic vector field $U=(U^i)\in
C^\infty(\beta^1_0M)$, put
$$
U^i_{|k}:=\left(\nabla_{\del{}{x^k}} U\right)^i,
\quad
U^i_{\cdot k}:=\left(\nabla_{\de{}{y^k}} U\right)^i.
$$

We have
$$
u_{|k}=\de u{x^k}-\Gamma^p_{kq}y^q\de u{y^p},
\quad
u_{\cdot k}=\de u{y^k},
$$
and
$$
U^i_{|k}=\de {U^i}{x^k}-\Gamma^p_{kq}y^q\de {U^i}{y^p}
+\Gamma^i_{kp}U^p,
\quad
U^i_{\cdot k}=\de{U^i}{y^k}.
$$

In the usual way, we extend these formulas to higher order
tensors:
\begin{multline*}
T^{i_1\dots i_r}_{j_1\dots j_s|k}
=\de{}{x^k}T^{i_1\dots i_r}_{j_1\dots j_s}
-\Gamma^p_{kq}y^q\de{}{y^p}T^{i_1\dots i_r}_{j_1\dots j_s}
\\
+\sum_{m=1}^r\Gamma^{i_m}_{kp} T^{i_1\dots i_{m-1}pi_{m+1}\dots
i_r}_{j_1\dots j_s} -\sum_{m=1}^s\Gamma^{p}_{kj_m} T^{i_1\dots
i_r}_{j_1\dots j_{m-1}pj_{m+1}\dots j_s}
\end{multline*}
and
$$
T^{i_1\dots i_r}_{j_1\dots j_s\cdot k}
=\de{}{y^k}T^{i_1\dots i_r}_{j_1\dots j_s}.
$$

We define the operators
$$
\nabla_{|}:C^\infty(\beta^r_sM)\to C^\infty(\beta^r_{s+1}M),
\quad
\nabla_{\cdot}:C^\infty(\beta^r_sM)\to C^\infty(\beta^r_{s+1}M)
$$
by
$$
(\nabla_{|} T)^{i_1\dots i_r}_{j_1\dots j_s k}
=\nabla_{|k} T^{i_1\dots i_r}_{j_1\dots j_{s}}
:=T^{i_1\dots i_r}_{j_1\dots j_s|k}
$$
and
$$
(\nabla_{\cdot} T)^{i_1\dots i_r}_{j_1\dots j_{s}k}
=\nabla_{\cdot k} T^{i_1\dots i_r}_{j_1\dots j_{s}}
=T^{i_1\dots i_r}_{j_1\dots j_s\cdot k}.
$$

For convenience, we also define $\nabla^{|}$ and $\nabla^{\cdot}$
by
$$\nabla^{|i}=g^{ij} \nabla_{|j},\quad
\nabla^{\cdot i}=g^{ij}\nabla_{\cdot j}.
$$


\subsection{Modified horizontal derivative for magnetic flows}
The form $\Omega$, regarded as an antisymmetric tensor field
$(\Omega_{ij})\in C^\infty(\tau^0_2 M)$, gives rise to a corresponding
semibasic tensor field. We
define the {\it Lorentz force} $Y\in C^\infty(\beta^1_1M)$ by
\begin{equation}\label{lorentz}
Y^i_j(x,y)=\Omega_{jk}(x)g^{ik}(x,y).
\end{equation}
We also define
$$
Y(U)=(Y^i_j U^j).
$$
Note that $Y$ is skew symmetric with respect to $g$:
$$
\langle Y(U),V\rangle=-\langle U,Y(V)\rangle.
$$
Straightforward calculations show that
\begin{equation}\label{spray}
\G(x,y)=y^i\del{}{x^i}+y^iY^j_i\de{}{y^j}.
\end{equation}

If $u\in C^\infty(TM\setminus\{0\})$, then by (\ref{spray})
$$
\G u(x,y)=y^i\left(\del u{x^i}+Y^j_i\de u{y^j}\right)
=y^i(u_{|i}+Y^j_iu_{\cdot j}).
$$

Suppose that for a smooth function $u:SM\to \mathbb R$
we have
$$
\G u=\varphi.
$$
Extend $u$ to a positively homogeneous function (of degree $0$)
on $TM\setminus\{0\}$, denoting the extension by $u$ again.

For $(x,y)\in TM$, define
$$
\mathbf X u=y^i(u_{|i}+FY^j_iu_{\cdot j}).
$$

Then on $TM\setminus\{0\}$
we have
$$
\mathbf Xu=\phi,
$$
where $\phi$ is the positively homogeneous extension of $\varphi$ to
$TM\setminus\{0\}$ of degree $1$.

Given $T=(T^{i_1\dots i_r}_{j_1\dots j_s})\in C^\infty(\beta^r_s M)$,
put
$$
T^{i_1\dots i_r}_{j_1\dots j_s: k}
=T^{i_1\dots i_r}_{j_1\dots j_s|k}
+FY^j_k T^{i_1\dots i_r}_{j_1\dots j_s\cdot j}.
$$

Finally, given $u\in
C^\infty(TM\setminus\{0\})$, define
$$
\nabla^{:}u=(u^{:i})=(g^{ij} u_{:j}).
$$

\subsection{An integral identity}A crucial element in our proofs is the following integral
version of the Pestov identity proved in \cite{DP2}:

\begin{multline}\label{pestov-integral-final}
\int_{SM}\big\{|\mathbf X (\nabla^{\cdot} u)|^2
-\langle \mathbf R_y(\nabla^{\cdot} u),\nabla^{\cdot} u\rangle
-L(Y(y),\nabla^{\cdot} u,\nabla^{\cdot}u)
-\langle\nabla^{\cdot } (\mathbf X u),Y(\nabla^{\cdot} u)\rangle
\\
-2\langle Y(y),\nabla^{\cdot}u\rangle^2
+\langle \nabla^{:} u,Y(\nabla^{\cdot}u)\rangle
+\langle\nabla_{|(\nabla^{\cdot}u)}Y(y),\nabla^{\cdot}u\rangle\big\}\,d\mu
\\
=\int_{SM}\big\{|\nabla^{\cdot} (\mathbf X  u)|^2
-n(\mathbf X u)^2\big\}\,d\mu.
\end{multline}
We note the following points:
\begin{enumerate}
\item $\nabla^{\cdot}u$ vanishes if and only if $u$ is the pull back of
a function on $M$;
\item ${\bf R}$ and $L$ are respectively the Riemann curvature operator and the
Landsberg tensor from Finsler geometry; $Y$
is the Lorentz force associated with the magnetic field;
\item $n$ is the dimension of $M$.
\end{enumerate}

We may regard the identity as a kind of ``dynamical Weitzenb\"ock
formula". We will also need the following lemma \cite[Lemma 4.4]{DP2}:

\begin{Lemma}\label{int-nabla}
Let $\phi\in C^\infty(TM\setminus\{0\})$ be such that
$\phi=\varphi_0 F+\psi$, where $\varphi_0$
is independent of $y$ while $\psi$ depends linearly on $y$. Then
$$
\int_{SM}|\nabla^{\cdot}\phi|^2\,d\mu
=\int_{SM}(\varphi_0^2+n\psi^2)\,d\mu.
$$
\end{Lemma}

In \cite{DP2} we dealt with the left hand side of (\ref{pestov-integral-final})
assuming that $\phi$ is Anosov. We will show in the next section, how to bypass
hyperbolicity just using the Green subbundles and a further calculation to show
that $\theta$ in Theorem A must be closed.

\section{Proof of Theorem A}

Define
$$
\mathcal C(Z)=\mathbf R_y(Z)-Y(\mathbf X Z)-(\nabla_{|Z} Y)(y).
$$

Then the following holds:
\begin{align*}
\langle \mathcal C(\nabla^{\cdot}u),\nabla^{\cdot}u\rangle
&=\langle \mathbf R_y(\nabla^{\cdot} u),\nabla^{\cdot} u\rangle
+\langle \mathbf X(\nabla^{\cdot}u),Y(\nabla^{\cdot} u)\rangle
-\langle(\nabla_{|(\nabla^{\cdot} u)}Y)(y),\nabla^{\cdot}u\rangle
\\
&=\langle \mathbf R_y(\nabla^{\cdot}u),\nabla^{\cdot} u\rangle
+\langle \nabla^{\cdot}(\mathbf X  u)-\nabla^{:}u
-\langle Y(y),\nabla^{\cdot}u\rangle y,
Y(\nabla^{\cdot}u)\rangle
\\
&\quad-\langle(\nabla_{(\nabla^{\cdot} u)}Y)(y),\nabla^{\cdot}u\rangle
\\
&=\langle \mathbf R_y(\nabla^{\cdot}u),\nabla^{\cdot} u\rangle
+\langle \nabla^{\cdot}(\mathbf X  u),
Y(\nabla^{\cdot}u)\rangle
-\langle\nabla^{:}u,Y(\nabla^{\cdot}u)\rangle
\\
&\quad+\langle Y(y),\nabla^{\cdot}u\rangle^2
-\langle(\nabla_{|(\nabla^{\cdot} u)}Y)(y),\nabla^{\cdot}u\rangle.
\end{align*}

Suppose $\G u=h\circ\pi+\theta$ and extend $u$
to a positively homogeneous function of degree zero on
$TM\setminus\{0\}$ (still denoted by $u$).  Then $\mathbf X(u)=F
h\circ\pi+\theta$. From
(\ref{pestov-integral-final}) and Lemma \ref{int-nabla}
we infer that
\begin{equation}
\int_{SM}\big\{|\mathbf X \nabla^{\cdot} u|^2
-\langle \mathcal C( \nabla^{\cdot} u), \nabla^{\cdot} u\rangle
-L(Y(y), \nabla^{\cdot} u, \nabla^{\cdot} u)
-\langle Y(y), \nabla^{\cdot} u\rangle^2\big\}\,d\mu
\label{menqc}
\end{equation}
\[\leq -\int_{SM}(h\circ\pi)^2\,d\mu.\]

For each $(x,y)\in SM$, let $\P(x,y):T_{x}M\to T_{x}M$ be the orthogonal projection
onto $\{y\}^{\perp}$ with respect to the fundamental tensor at $(x,y)$.

\begin{Theorem} Assume $SM$ has no conjugate points and let $Z(x,y)$ be a semibasic
vector field with $\langle Z(x,y),y\rangle=0$ for all $(x,y)\in SM$. Let $S(x,y):T_{x}M\to T_{x}M$
be a linear map whose graph is one of the Green subbundles.

Then
\[\int_{SM}|\P(\mathbf X Z-S(Z))|^2\,d\mu=
\int_{SM}\big\{|\mathbf X Z|^2
-\langle \mathcal C( Z), Z\rangle
-L(Y(y), Z, Z)
-\langle Y(y), Z\rangle^2\big\}\,d\mu.\]
Moreover
\[\int_{SM}|\P(\mathbf X Z-S(Z))|^2\,d\mu=0\]
if and only if
\[Z_{0}(t):=Z(\phi_{t}(x,y))+\left(\int_{0}^{t}f(\phi_{s}(x,y))\,ds\right)\,\dot{\ga}(t)\]
is a magnetic Jacobi field along $\ga(t)=\pi\circ\phi_{t}(x,y)$ for all
$(x,y)\in SM$, where 
\[f(x,y):=\langle Z(x,y),Y_{x}(y)\rangle.\]

\label{new}
\end{Theorem}

\begin{proof} Fix $v\in SM$ and $T>0$ and let $\gamma$ be the unit speed magnetic geodesic
determined by $v$.
Let
\[\I:=\int_0^T \big\{|\dot Z|^2-\langle \mathcal C(Z),Z\rangle
-L(Y(\dot\gamma),Z,Z)
-\langle Y(\dot\gamma),Z\rangle^2 \big\}\,dt.\]
Since
$$
\langle \ddot Z,Z\rangle=D_{\dot\gamma}(\langle \dot Z,Z\rangle)
-|\dot Z|^2
+\langle (\nabla_{\cdot \dot Z}Y)(\dot\gamma),Z\rangle,
$$
we have
\[\I=
\langle \dot{Z},Z\rangle\bigg|_{0}^{T}-\int_{0}^{T}\{\left\langle {\mathcal A}(Z),Z\right\rangle+\langle Y(\dot{\gamma}),Z\rangle^{2}\}\,dt,\]
where
\begin{align*}
{\mathcal A}(Z)&=\ddot{Z}+\mathbf R_{\dot{\gamma}}(Z)
-Y(\dot{Z})-(\nabla_{|Z}Y)(\dot{\gamma}) -(\nabla_{\cdot \dot
Z}Y)(\dot\gamma)-L(Z,Y(\dot\gamma))   \label{A}
\\
&=\ddot Z+\mathcal C(Z) -(\nabla_{\cdot \dot
Z}Y)(\dot\gamma)-L(Z,Y(\dot\gamma)).\nonumber
\end{align*}

If $\xi\in E(v)$, then $J_{\xi}(t)=d\pi\circ d\phi_{t}(\xi)$ satisfies the Jacobi equation
${\mathcal A}(J_{i})=0$. 

Since for all $t\in \mathbb R$,
\[\left. d\pi_{\dot{\gamma}(t)}\right|_{E(\dot{\gamma}(t))}:E(\dot{\gamma}(t))\rightarrow T_{{\gamma}(t)}M\]
is an isomorphism, there exists a basis $\{\xi_{1},\dots,\xi_{n}\}$ of $E(v)$ such
that $\{J_{\xi_{1}}(t),\dots,J_{\xi_{n}}(t)\}$ is a basis of $T_{{\gamma}(t)}M$ for all $t\in \mathbb R$.
Without loss of generality we may assume that $\xi_{1}=(v,S(v))$ and $J_{\xi_{1}}=\dot{\gamma}$.

Let us set for brevity $J_{i}=J_{\xi_{i}}$. Then we can write
\[ Z(t)=\sum_{i=1}^{n}f_{i}(t)J_{i}(t),\]
for some smooth functions $f_{1},\dots,f_{n}$ and thus,
\begin{equation}
\I=\langle\dot{Z},Z\rangle\bigg|_{0}^{T}-\sum_{i,j}\int_{0}^{T}\left\langle{\mathcal A}(f_{i}J_{i}),f_{j}J_{j}\right\rangle\,dt
-\int_{0}^{T}\langle Y(\dot{\gamma}),Z\rangle^{2}\,dt.      \label{bilinear}
\end{equation}
An easy computation shows that
\[{\mathcal A}(f_{i}J_{i})=\ddot{f}_{i}J_{i}
+2\dot{f}_{i}\dot{J}_{i}-\dot{f}_{i}Y(J_{i})
-\dot f_i(\nabla_{\cdot J_i}Y)(\dot\gamma)
+f_{i}{\mathcal A}(J_{i}).\]

Indeed,
$$
D_{\dot\gamma}D_{\dot\gamma}(f_iJ_i)
=\ddot f_iJ_i+2\dot f_i\dot J_i+f_i\ddot J_i,
$$
$$
\mathbf R_{\dot\gamma}(f_iJ_i)=f_i\mathbf R_{\dot\gamma}(J_i),
$$
$$
Y(D_{\dot\gamma}(f_iJ_i))
=\dot f_i Y(J_i)+f_i Y(\dot J_i),
$$
$$
(\nabla_{|f_iJ_i}Y)(\dot\gamma)=f_i(\nabla_{|J_i}Y)(\dot\gamma),
$$
$$
(\nabla_{\cdot D_{\dot\gamma}(f_iJ_i)}Y)(\dot\gamma)
=\dot f_i(\nabla_{\cdot J_i}Y)(\dot\gamma)
+f_i(\nabla_{\cdot \dot J_i}Y)(\dot\gamma),
$$
$$
\mathbf L(f_iJ_i,Y(\dot\gamma))=f_i\mathbf L(J_i,Y(\dot\gamma)).
$$

Since $J_{i}$ satisfies the Jacobi equation
${\mathcal A}(J_{i})=0$ we have
\[\langle{\mathcal A}(f_{i}J_{i}),J_{j}\rangle
=\ddot{f}_{i}\langle J_{i},J_{j}\rangle
+2\dot{f}_{i}\langle\dot{J}_{i},J_{j}\rangle
-\dot{f}_{i}\langle Y(J_{i}),J_{j}\rangle
-\dot f_i\langle(\nabla_{\cdot J_i}Y)(\dot\gamma),J_j\rangle.\]

Observe that since $E$ is a Lagrangian subspace,
\[\langle J_{i},\dot{J}_{j}\rangle
-\langle\dot{J}_{i},J_{j}\rangle
+\langle Y(J_{i}),J_{j}\rangle=0,\]
and then
\[\langle{\mathcal A}(f_{i}J_{i}),J_{j}\rangle
=\frac{d}{dt}(\dot{f}_{i}\langle J_{i},J_{j}\rangle).\]
Now we can write
\[\int_{0}^{T}\langle{\mathcal A}(f_{i}J_{i}),f_{j}J_{j}\rangle\,dt
=\left.\langle\dot{f}_{i}J_{i},f_{j}J_{j}\rangle\right|_{0}^{T}
-\int_{0}^{T}\langle\dot{f}_{i}J_{i},\dot{f}_{j}J_{j}\rangle\,dt.\]
Combining the last equality with (\ref{bilinear}) we obtain
\[\I=\int_{0}^{T}\bigg|\sum_{i=1}^{n}\dot{f}_{i}J_{i}\bigg|^{2}\,dt
-\bigg\langle\sum_{i=1}^{n}\dot{f}_{i}J_{i}-\dot{Z},Z\bigg\rangle\bigg|_{0}^{T}
-\int_{0}^{T}\langle Y(\dot{\gamma}),Z\rangle^{2}\,dt.\]
Note that $\dot{J}_{i}(t)=S_{\dot{\gamma}(t)}J_{i}(t)$, hence
\[\sum_{i=1}^{n}f_{i}\dot{J}_{i}=S\left(\sum_{i=1}^{n}f_{i}J_{i}\right)=S(Z),\]
which implies together with $\dot{Z}=\sum_{i=1}^{n}\dot{f}_{i}J_{i}+\sum_{i=1}^{n}f_{i}\dot{J}_{i}$
that
\begin{equation}
\I=\int_{0}^{T}|\dot{Z}-S(Z)| ^{2}\,dt+\langle S(Z),Z\rangle\bigg|_{0}^{T}
-\int_{0}^{T}\langle Y(\dot{\gamma}),Z\rangle^{2}\,dt.
\label{finvi}
\end{equation}
Now let
\[W:=\sum_{i=2}^{n}\dot{f}_{i}J_{i}.\]
Since $J_1=\dot{\gamma}$ we have:
\[\bigg\langle
\sum_{i=1}^{n}\dot{f}_{i}J_{i},\sum_{i=1}^{n}\dot{f}_{i}J_{i}\bigg\rangle
=\langle \dot{f}_{1}\dot{\gamma}+W, \dot{f}_{1}\dot{\gamma}+W\rangle
=\dot{f}_{1}^{2}+2\dot{f}_{1}\langle \dot{\gamma},W\rangle+\langle W,W\rangle.\]
Differentiating $\langle Z,\dot{\gamma}\rangle=0$ we get
\[\langle \dot{Z},\dot{\gamma}\rangle+\langle Z,Y(\dot{\gamma})\rangle=0.\]
But
\[\langle \dot{Z},\dot{\gamma}\rangle=\left\langle \sum_{i=1}^{n}\dot{f}_{i}J_{i},\dot{\gamma}\right\rangle
=\dot{f}_{1}+\langle W,\dot{\gamma}\rangle\]
since $\langle \dot{J}_{i},\dot{\gamma}\rangle=0$ for all $i$. Therefore
\[\langle Y(\dot{\gamma}),Z\rangle^{2}=\dot{f}_{1}^{2}+2\dot{f}_{1}\langle W,\dot{\gamma}\rangle+
\langle W,\dot{\gamma}\rangle^{2}.\]
Thus
\[\left\langle \sum_{i=1}^{n}\dot{f}_{i}J_{i},\sum_{i=1}^{n}\dot{f}_{i}J_{i}\right\rangle
-\langle Y(\dot{\gamma}),Z\rangle^{2}=\langle W,W\rangle-\langle W,\dot{\gamma}\rangle^{2}.\]
If we let $W^{\perp}$ be the orthogonal projection of $W$ to $\dot{\gamma}^{\perp}$, the last equation
and (\ref{finvi}) give:
\[\I=\int_{0}^{T}| W^{\perp}|^{2}\,dt+\langle S(Z),Z\rangle\bigg|_{0}^{T}.\]
Observe now that
\[\P(\dot{Z}-S(Z))=\P(W)=W^{\perp}\]
thus
\[\I=\int_{0}^{T}|\P(\dot{Z}-S(Z))|^{2}\,dt+\langle S(Z),Z\rangle\bigg|_{0}^{T}\]
and we have established the equality
\begin{multline}\label{alongorb}
\int_0^T \big\{|\mathbf X Z|^2-\langle \mathcal C(Z),Z\rangle
-L(Y(\dot\gamma),Z,Z)
-\langle Y(\dot\gamma),Z\rangle^2 \big\}\,dt\\
=\int_{0}^{T}|\P(\mathbf X Z-S(Z))|^{2}\,dt+\langle S(Z),Z\rangle\bigg|_{0}^{T}
\end{multline}

We now set $T=1$ in (\ref{alongorb}) and we integrate the equality
with respect to the Liouville measure $\mu$. Since $\phi$ preserves $\mu$
we have
\[\int_{SM}\langle S(Z)(\phi_{1}(x,y)),Z(\phi_{1}(x,y))\rangle\,d\mu=
\int_{SM}\langle S(Z)(x,y),Z(x,y)\rangle\,d\mu.\]
Thus, using Fubini's theorem we obtain
\[\int_{SM}|\P(\mathbf X Z-S(Z))|^2\,d\mu=
\int_{SM}\big\{|\mathbf X Z|^2
-\langle \mathcal C( Z), Z\rangle
-L(Y(y), Z, Z)
-\langle Y(y), Z\rangle^2\big\}\,d\mu.\]

Suppose now
\[\int_{SM}|\P(\mathbf X Z-S(Z))|^2\,d\mu=0.\]
Using Fubini's theorem again we have for any $T\in\re$
\[\int_{SM}\left(\int_{-T}^{T}|\P(\mathbf X Z-S(Z))|^2(\phi_{t}(x,y))\,dt\right)\,d\mu(x,y)=0.\]
Since $S$ is smooth along the flow we conclude that for all $t\in\re$
\[\P(\mathbf X Z-S(Z))(\phi_{t}(x,y))=0\]
for almost every $(x,y)\in SM$. Let
\[Z_{0}(t):=Z(\phi_{t}(x,y))+x(t)\dot{\ga}(t)\]
where
\[x(t):=\int_{0}^{t}f(\phi_{s}(x,y))\,ds.\]
Now observe that
\[\dot{Z}_{0}=\dot{Z}+xY(\dot{\ga})+\dot{x}\dot{\ga},\]
\[SZ_{0}=SZ+xY(\dot{\ga}),\]
hence
\[\dot{Z}-SZ+\dot{x}\dot{\ga}=\dot{Z}_{0}-SZ_{0}.\]
Now note that since $\langle Z,\dot{\ga}\rangle=0$,
$\dot{x}=\langle Z,Y(\dot{\ga})\rangle=-\langle \dot{Z},\dot{\ga}\rangle$.
Also $\langle S(x,y)(z),y\rangle=0$ for any $(x,y)\in SM$ and $z\in T_{x}M$, since
the Green subbundle is contained in $T(SM)$.
It follows that
\[\langle \dot{Z}_{0}-SZ_{0}, \dot{\ga}\rangle=0\]
and thus
\[0=\P(\dot{Z}-SZ+\dot{x}\dot{\ga})=\P(\dot{Z}_{0}-SZ_{0})=
\dot{Z}_{0}-SZ_{0}.\]
Hence $Z_{0}$ is a Jacobi field along $\ga$ for almost every $(x,y)\in SM$.
This implies
\[0=\mathcal A (Z_{0})=\mathcal A(Z)+\mathcal A(x\dot{\ga}).\]
An easy calculation shows that
\[\mathcal A(x\dot{\ga})=\frac{D}{dt}\mathbf(f\dot{\ga})\]
therefore
\[\mathcal A(Z)+\mathbf X(f\dot{\ga})=0\]
for all $t\in \re$ and almost every $(x,y)\in SM$. 
Consider now the operator acting on semibasic vector fields $V$ given by
\[V\mapsto \mathcal A(V)+\mathbf X(\langle V,Y(y)\rangle\,y)\]
This operator annihilates $Z$ for almost every $(x,y)\in SM$ and since $Z$ is smooth
it must annihilate $Z$ for {\it every} $(x,y)\in SM$. Going backwards we now deduce
that $Z_{0}(t)$ is a Jacobi field along {\it every} magnetic geodesic $\ga$ as desired.

Conversely, it is now easy to check that if $Z_{0}(t)$ is a Jacobi field along every magnetic geodesic,
then
\[\int_{SM}|\P(\mathbf X Z-S(Z))|^2\,d\mu=0.\]

\end{proof}

\subsection{Proof of Theorem A} First observe that (\ref{menqc}) and Theorem \ref{new} imply right away that
$h=0$.
Let us show that $\theta$ must be closed.

Note that (\ref{menqc}) and Theorem \ref{new} also imply that $Z=\nabla^\cdot u$ satisfies the equation
\begin{equation}\label{jacobi-global}
\mathcal A(Z)+\mathbf X(fy)=0,
\end{equation}
with 
$$
f(x,y)=\langle Z(x,y),Y_x(y)\rangle,
$$
\begin{equation}
\mathcal A(Z)=\mathbf X^2 Z+\mathbf R_y(Z)
-Y(\mathbf X Z)-(\nabla_{|Z}Y)(y) 
-(\nabla_{\cdot \mathbf X Z}Y)(y)-\mathbf L(Z,Y(y)).   \label{A}
\end{equation}

Note that (see \cite[Proof of Lemma 4.7]{DP2})
\begin{equation*}
\mathbf X(\nabla^{\cdot}u)
=\nabla^{\cdot}(\mathbf Xu)
-\nabla^{:}u
-\langle Y(y),\nabla^{\cdot}u\rangle y.
\end{equation*}

Therefore,
\begin{align}\label{x2z}
\mathbf X^2 Z&=\mathbf X[\nabla^{\cdot}(\mathbf Xu)-\nabla^{:}u -f y]\nonumber
\\
&=\mathbf X[\nabla^{\cdot}(\mathbf Xu)]-\mathbf X(\nabla^{:}u) -\mathbf X(f y).
\end{align}

We have
\begin{align}\label{xnu}
\mathbf X(\nabla^{:}u)&=y^k( g^{ij}u_{:j})_{:k}
=y^k\left(g^{ij}_{:k}u_{:j}+g^{ij} u_{:j:k}\right)\nonumber
\\
&=-2y^k Y^s_k g^{il}g^{jm}C_{lms}u_{:j}
+y^kg^{ij} \left[u_{:k:j}-(u_{:k:j}-u_{:j:k})\right]\nonumber
\\
&=-2\mathbf C(\nabla^{:} u, Y(y))
+g^{ij}[(y^ku_{:k})_{:j}-y^k_{:j}u_{:k}]
-y^k g^{ij}\tilde R^s_{kj}u_{\cdot s}\nonumber
\\
&=-2\mathbf C(\nabla^{:} u, Y(y))+\nabla^{:}(\mathbf X u)
-g^{ij}Y^k_j u_{:k}-y^k g^{ij}\tilde R^s_{kj}u_{\cdot s}\nonumber
\\
&=\nabla^{:}(\mathbf X u)+Y(\nabla^: u)-2\mathbf C(\nabla^{:} u, Y(y))
-y^k g^{ij}\tilde R^s_{kj}u_{\cdot s},
\end{align}
where $C$ is the Cartan tensor of $F$ and $\mathbf C(U,V)=\left(g^{il}C_{lms} U^mV^s\right)$.

Next,
\begin{equation}\label{tilde-r}
\begin{aligned}
y^k g^{ij}\tilde R^s_{kj}u_{\cdot s}&=y^k g^{ij}\left[R^s_{kj}+(Y^s_{k|j}-Y^s_{j|k})
-(P^s_{km}Y^m_j-P^s_{jm}Y^m_k)\right.
\\
&\quad +\left.(Y^m_kY^s_{j\cdot m}-Y^m_jY^s_{k\cdot m})
+y_m(Y^m_j Y^s_k-Y^m_k Y^s_j)\right] u_{\cdot s}.
\end{aligned}
\end{equation}

Note that
\begin{equation}\label{r}
y^k g^{ij} R^s_{kj}u_{\cdot s}=-\mathbf R_y(Z),
\end{equation}
\begin{align}\label{y-y}
y^k g^{ij}(Y^s_{k|j}-Y^s_{j|k})u_{\cdot s}
&=y^k g^{ij}\left((\Omega_{km}g^{sm})_{|j}
-(\Omega_{jm}g^{sm})_{|k}\right)u_{\cdot s}\nonumber
\\
&=y^k g^{ij}\left(\Omega_{km,j}g^{sm}
-\Omega_{jm,k}g^{sm}\right)u_{\cdot s}\nonumber
\\
&=-y^k g^{ij}\Omega_{kj,m}g^{sm}u_{\cdot s}=-y^k Y^i_{k|m}g^{sm}u_{\cdot s}\nonumber
\\
&=-(\nabla_{Z}Y)(y)
\end{align}
in view of the identity $\Omega_{km,j}+\Omega_{mj,k}+\Omega_{jk,m}=0$
($\Omega$ is closed),
\begin{equation}\label{p-p}
y^k g^{ij}\left(P^s_{km}Y^m_j-P^s_{jm}Y^m_k\right) u_{\cdot s}=\mathbf L(Z,Y(y))
\end{equation}
in view of \cite[(20) and (24)]{DP2},
\begin{multline}\label{yy}
y^k g^{ij}\left(Y^m_kY^s_{j\cdot m}-Y^m_jY^s_{k\cdot m}\right)u_{\cdot s}
\\
=-2y^k g^{ij}\left(Y^m_k Y^n_jg^{sl}C_{lnm}
-Y^m_jY^n_kg^{sl}C_{lnm}\right)u_{\cdot s}=0
\end{multline}
in view of \cite[(32)]{DP2} and the symmetry of $C$, and
\begin{equation}\label{yyy}
y^k g^{ij}y_m\left(Y^m_j Y^s_k-Y^m_k Y^s_j\right) u_{\cdot s}=-\langle Y(y),Z\rangle Y(y)
\end{equation}
by the skew symmetry of $Y$.

Using \eqref{xnu}--\eqref{yyy} in \eqref{x2z}, we obtain
\begin{equation}\label{x2zf}
\begin{aligned}
\mathbf X^2 Z&=\mathbf X[\nabla^{\cdot}(\mathbf Xu)]
-\nabla^{:}(\mathbf X u)-Y(\nabla^: u)+2\mathbf C(\nabla^{:} u, Y(y))
\\
&\quad -\mathbf R_y(Z)
+(\nabla_{Z}Y)(y)
-\mathbf L(Z,Y(y))
\\
&\quad-\langle Y(y),Z\rangle Y(y)
-\mathbf X(f y).
\end{aligned}
\end{equation}

Also,
\begin{align}
Y(\mathbf X Z)&=Y(\mathbf X(\nabla^\cdot u))
=Y\left[\nabla^{\cdot}(\mathbf Xu)
-\nabla^{:}u
-\langle Y(y),\nabla^{\cdot}u\rangle y\right]\nonumber
\\
&=Y(\nabla^{\cdot}(\mathbf Xu))\label{yxz}
-Y(\nabla^{:}u)
-\langle Y(y),Z\rangle Y(y),
\end{align}

\begin{align}
(\nabla_{\cdot \mathbf X Z}Y)(y)&=\left[(\mathbf X u)^{\cdot k}
-u^{:k}-\langle Y(y),\nabla^{\cdot}u\rangle y^k\right]y^jY^i_{j\cdot k}\nonumber
\\
&=-2\left[(\mathbf X u)^{\cdot k}
-u^{:k}-\langle Y(y),\nabla^{\cdot}u\rangle y^k\right]y^j
Y^m_j g^{il}C_{lmk}\nonumber
\\
&=-2\mathbf C(\nabla^\cdot(\mathbf X u),Y(y))
+2\mathbf C(\nabla^{:}u,Y(y)).\label{nxzy}
\end{align}

Using \eqref{A} and \eqref{x2zf}--\eqref{nxzy} in \eqref{jacobi-global}
and performing cancelations,
we deduce the following:
\begin{equation}\label{prefinal}
\mathbf X[\nabla^{\cdot}(\mathbf Xu)]
-\nabla^{:}(\mathbf X u)
-Y(\nabla^{\cdot}(\mathbf Xu))
+2\mathbf C(\nabla^\cdot(\mathbf X u),Y(y))=0.
\end{equation}

Since $\mathbf X u(x,y)=\theta_i(x)y^i$, we have
\begin{align*}
&\mathbf X[\nabla^{\cdot}(\mathbf Xu)]
-\nabla^{:}(\mathbf X u)
-Y(\nabla^{\cdot}(\mathbf Xu))
+2\mathbf C(\nabla^\cdot(\mathbf X u),Y(y))
\\
&\quad=y^k (g^{ij}\theta_i)_{:k}-g^{ij}(\theta_ky^k)_{:j}-Y^i_jg^{jk}\theta_k
+2g^{il}C_{lms}g^{mj}\theta_jY^s_k y^k
\\
&\quad=y^k(g^{ij}_{:k}\theta_j+g^{ij}\theta_{j,k})
-g^{ij}\theta_{k,j}y^k-g^{ij}\theta_ky^k_{:j}-Y^i_jg^{jk}\theta_k
+2y^kY^s_k g^{il}g^{jm}C_{lms}\theta_j
\\
&\quad=-2y^k Y^s_k g^{il}g^{jm}C_{lms}\theta_j +g^{ij}\theta_{j,k}y^k
-g^{ij}\theta_{k,j}y^k-g^{ij}\theta_kY^k_j-g^{jk}Y^i_j\theta_k
\\
&\qquad+2y^kY^s_k g^{il}g^{jm}C_{lms}\theta_j
\\
&\quad=g^{ij}(\theta_{j,k}-\theta_{k,j})y^k,
\end{align*}

Now, \eqref{prefinal} yields
$$
g^{ij}(\theta_{j,k}-\theta_{k,j})y^k=0,
$$
which means that the form $\theta$ is closed.

\qed

\begin{Remark}{\rm Observe that if we assume that the hyperbolic closed orbits of $\phi$
are dense, it is easy to show that $\theta$ is exact directly from (\ref{menqc}) and Theorem \ref{new}.  
Indeed we have
\[\int_{SM}|\P(\mathbf X Z-S(Z))|^2\,d\mu=0\]
where $Z=\nabla^{\cdot}u$. Since $Z$ is bounded, the Jacobi field
$Z_{0}$ grows at most linearly. If $(x,y)$ gives rise to a hyperbolic
closed orbit, then $Z$ must vanish along it. Since we are assuming
that the hyperbolic closed orbits are dense, $Z$ must vanish everywhere
on $SM$. It follows that $u(x,y)$ depends only on $x$ and $\theta$ is exact.

Also note that the set $Q:=\{(x,y)\in SM:\;Z(x,y)=0\}$ is $\phi$-invariant
and $\pi(Q)=M$ (cf. end of the proof of Theorem B).

}

\end{Remark}

\section{Invariant distributions and asymptotic cycles}
\label{dt}

In \cite{FF} L. Flaminio and G. Forni studied the cohomological equation for the
horocycle flow $\phi^h$ of a compact hyperbolic surface. Let $U$ be the vector field
generating $\phi^h$. They showed (among several other results) that the equation
$U(u)=v$ for $v\in C^{\infty}(SM)$ admits a $C^{\infty}$ solution $u$ if and only
if $\mathcal D(v)=0$ for every invariant distribution $\mathcal D$ of the horocycle flow.
Recall that an invariant distribution $\mathcal D$ is an element of the dual space of $C^{\infty}(SM)$
such that $\mathcal D(U(u))=0$ for every smooth function $u$. Let $\mathcal I(\phi^h)$
be the space of all invariant distributions. Flaminio and Forni also show that
$\mathcal I(\phi^h)$ is a vector space of infinite countable dimension
completely determined by the spectrum $\sigma$ of the Laplacian of $M$ and
the genus of $M$ as follows (their result also provides precise information about the Sobolev
regularity of the invariant distributions):

\[\mathcal I(\phi^h)=\bigoplus_{\mu\in\sigma}\mathcal I_{\mu}\oplus
\bigoplus_{n\in \Z^+}\mathcal I_{n}\]
where
\begin{itemize}
\item for $\mu=0$, $\mathcal I_{0}$ is spanned by the $PU(1,1)$-invariant volume;
\item for $\mu>0$, $\mathcal I_{\mu}$ has dimension equal to twice the multiplicity of
$\mu\in\sigma$;
\item for $n\in\Z^{+}$, the space $\mathcal I_{n}$ has dimension equal to twice the rank
of the space of holomorphic sections of the $n$-th power of the canonical line bundle over
$M$.
\end{itemize}

There is also an explicit calculation of $\mathcal I_{\mu}$ and $\mathcal I_{n}$ in terms
of appropriate bases of the subspaces of the irreducible representations.

Let $X$ be the infinitesimal generator of the geodesic flow and let $V$ be the infinitesimal
generator of the action of $S^1$ on the fibres of $SM\to M$. Let $H$ be the vector field
associated with the flow $R^{-1}\circ g_{t}\circ R$, where $g_t$ is the geodesic flow and
$R(x,y)=(x,iy)$. Using the basis $\{X,H,V\}$ we can write:
\[U=-H+V;\]
\[\G=X+V.\]

Define (cf. \cite{GK1,FF}):
\[\eta_{+}:=X-i\,H\]
and
\[\eta_{-}:=X+i\,H.\]

Let $L^2(SM)$ be the space of square integrable functions with respect
to the Liouville measure of $SM$.
The space $L^{2}(SM)$ decomposes into an orthogonal direct sum of
subspaces $\sum H_{n}$, $n\in\Z$, such that on $H_{n}$, $-i\,V$ is
$n$ times the identity operator;
$\eta_{+}$ extends to a densely defined operator from
$H_{n}$ to $H_{n+1}$ for all $n$. Moreover, its transpose is
$-\eta_{-}$.

An orthogonal basis of the representation of the principal or complementary series is given by
\[\dots,\eta_{-}^k f,\dots,\eta_{-}^2 f,\eta_{-}f,f,\eta_{+}f,\eta_{+}^2 f,\dots,\eta_{+}^k f,\dots\]
where $f$ is a normalized eigenfunction of the Laplacian on $M$ ($V(f)=0$).

On the other hand an orthogonal basis of the representation of the holomorphic discrete series $\pi^+_{n}$
is given by
\[f,\eta_{+}f,\eta_{+}^2 f,\dots,\eta_{+}^k f,\dots\]
where $f\in H_{n}$ and $\eta_{-}f=0$ (and also with unit norm). 
Similarly for the anti-holomorphic discrete series $\pi^{-}_{n}$.

Given a smooth 1-form $\theta$ we can decompose $\theta$ as
\[\theta=\theta_{-1}+\theta_{1}\]
where
\[2\theta_{-1}=\theta+i V(\theta),\]
\[2\theta_{1}=\theta-i V(\theta).\]
Clearly $\theta_{\pm 1}\in H_{\pm 1}$.

The following lemma is straightforward:

\begin{Lemma} The form $\theta$ is closed if and only if $\Im \eta_{-}\theta_{1}=0$.
The form $\theta$ is coclosed if and only if $\Re \eta_{-}\theta_{1}=0$.
Also, $\theta$ is closed if and only $V(\theta)$ coclosed {\rm(}and hence $\theta$ is coclosed
if and only if $V(\theta)$ is closed since $V^2(\theta)=-\theta${\rm)}.
\label{util}
\end{Lemma}

We shall now give a proof of Theorem B for $n=1$ without using the integral Pestov identity and at the same
time we will compute the map $\rho:\mathcal I(\phi)\to H_{1}(M,\re)$. We will suppose that $h=0$ to simplify matters,
but the calculation below can be easily extended.

The first thing to observe is that $\mathcal D(\theta)=0$ for all $\mathcal D\in\mathcal I(\phi)$
iff $\mathcal D(V(\theta))=0$ for all $\mathcal D\in\mathcal I(\phi^h)$. This will allow us to use directly
the information on $\mathcal I(\phi^h)$ from \cite{FF}.

Note that we always have $\mathcal D(\theta)=0$ for $\mathcal D\in \mathcal I_{n}$
for $n\neq 1$.

Let us examine the component of $\theta$ in one of the spaces of the irreducible representation
corresponding to the principal or complementary series (where $\mathcal I_{\mu}$ acts).
Call it $\theta_{\mu}$. The expansion of this component in the basis described
above is
\[\theta_{\mu}=a_{\mu}\eta_{-}f_{\mu}+b_{\mu}\eta_{+}f_{\mu},\]
where
\[a_{\mu}\|\eta_{-}f_{\mu}\|^2=\langle\theta_{-1},\eta_{-}f_{\mu}\rangle=
-\langle \eta_{+}\theta_{-1},f_{\mu}\rangle,\]
\[b_{\mu}\|\eta_{+}f_{\mu}\|^2=\langle\theta_{1},\eta_{+}f_{\mu}\rangle=-\langle \eta_{-}\theta_{1},f_{\mu}\rangle\]
and it is easy to see that $\|\eta_{-}f_{\mu}\|^2=\|\eta_{+}f_{\mu}\|^2=\mu$.

Now, the space $\mathcal I_{\mu}$ is generated by certain distributions $\mathcal D^+_{\mu}$ and $\mathcal D^-_{\mu}$
that Flaminio and Forni compute explicitly in terms of their Fourier coefficients. From \cite[Section 3]{FF}
we see that these distributions have the property
\[\mathcal D^{\pm}_{\mu}(\eta_{-}f_{\mu})=\mathcal D^{\pm}_{\mu}(\eta_{+}f_{\mu}) \]
and call this common non-zero value value $r_{\pm}$.
Then

\begin{equation}
\mathcal D^{\pm}_{\mu}(\theta)=
\frac{-r_{\pm}}{\mu}\langle \eta_{+}\theta_{-1}+\eta_{-}\theta_{1},f_{\mu}\rangle=
\frac{-2r_{\pm}}{\mu}\langle \Re(\eta_{-}\theta_{1}),f_{\mu}\rangle.
\label{accion}
\end{equation}

In the case of the holomorphic discrete series we can do a similar calculation. We can think of each
$f_{1}$ with $\eta_{-}f_{1}=0$ as a holomorphic section of the canonical line bundle over $M$ which gives
rise to an invariant distribution
$\mathcal D^1\in \mathcal I_{1}$ and
\begin{equation}
\mathcal D^1(\theta)=\langle \theta_{1},f_{1}\rangle \mathcal D^1(f_{1})
\label{dis}
\end{equation}
with $\mathcal D^1(f_{1})\neq 0$ and similarly for the anti-holomorphic discrete series.

Suppose now $\G(u)=\theta$ and let us show that $\theta$ must be exact.
Without loss of generality we can assume that $\theta$ is coclosed, since we can always
write $\theta=\theta'+df$, where $\theta'$ is coclosed and $\G(u-f\circ\pi)=\theta'$.
Then $\mathcal D(V(\theta))=0$ for all $\mathcal D\in \mathcal I(\phi^h)$.
Hence by (\ref{accion})
\[\langle \Re(\eta_{-}V(\theta)_{1}),f_{\mu}\rangle=0\]
for all $\mu\in\sigma$ which implies $\Re(\eta_{-}V(\theta)_{1})=0$. By Lemma \ref{util},
$\theta$ must be closed and hence harmonic (since we are assuming it is coclosed) and so is
$V(\theta)$.
Thus $\eta_{-}V(\theta)_{1}=0$ and we have an invariant distribution associated
with $V(\theta)_1$. But using (\ref{dis}) we see that $\theta$ must vanish identically.

This argument also shows the following. Take $[\omega]\in H^{1}(M,\re)$ and represent
the class by a harmonic 1-form $\omega$. Equation (\ref{accion}) shows that
$\mathcal D(\omega)=0$ for any $\mathcal D\in \mathcal I(\phi)$
corresponding to a distribution from $\mathcal I_{\mu}$.
On the other hand by (\ref{dis}) the restriction of $\rho$ to the subspace corresponding
to $\mathcal I_{1}$ will be an isomorphism onto $H_{1}(M,\re)$.

\medskip

We now observe that the results in \cite{FF} also give complete information for higher order tensors.
Let $C^{\infty}_{n}(SM)=H_{n}\cap C^{\infty}(SM)$. Note that if $u\in C^{\infty}_{n}(SM)$,
then $\G(u)\in C^{\infty}_{n+1}(SM)\oplus C^{\infty}_{n}(SM)\oplus C^{\infty}_{n-1}(SM)$.
Set
\[F^{\infty}_{k}(SM):=\bigoplus_{|i|\leq k}C^{\infty}_{i}(SM).\]

\begin{Definition} {\rm The magnetic ray transform is the function $I: C^{\infty}(SM)\to (\mathcal I(\phi))^*$
given by
\[I(v)(\mathcal D)=\mathcal D(v).\]

}
\end{Definition}

\begin{Proposition} For any $k\geq 0$, the kernel of $I$ restricted to $F^{\infty}_{k}(SM)$ is $\G(F_{k-1}^{\infty}(SM))$.
{\rm(}If $k=0$, we interpret this as saying that $I$ is injective.{\rm)}
\label{higher}
\end{Proposition}

\begin{proof} Take $v$ in the kernel of $I$. We know by Theorem \ref{flafor} that there exists a smooth $u$
(unique up to addition of a constant) such that $\G(u)=v$. We must show that $u\in F_{k-1}^{\infty}(SM)$.

Using the conjugacy between $\phi$ and $\phi^h$ it is easy to see that it suffices to show this claim
for the classical horocycle flow $U$. Indeed, note that $\mathcal D(v)=0$ for all $\mathcal D\in\mathcal I(\phi)$
iff $\mathcal D((x,y)\mapsto v(x,iy))=0$ for all $\mathcal D\in\mathcal I(\phi^h)$.

But in \cite[Section 4]{FF}, there are explicit formulas for the Fourier coefficients of $u$ in terms of those
of $v$ for each of the possible representations, where $U(u)=v$.
 For example, for the principal series (and $\nu\neq 0$) one obtains:
\[u_{n}=\frac{2i}{\nu}\sum_{l<n}\left(\frac{\Pi_{\nu,|l|}}{\Pi_{\nu,|n|}}-1\right)v_{l}=
\frac{-2i}{\nu}\sum_{l\geq n}\left(\frac{\Pi_{\nu,|l|}}{\Pi_{\nu,|n|}}-1\right)v_{l},\]
where $\nu$ is related to an eigenvalue $\mu$ of the Laplacian by $1-\nu^2=4\mu$, and $\Pi_{\nu,l}$
are certain coefficients defined in \cite[Section 2]{FF} whose precise value is of no importance to us.
Now, if $v\in F^{\infty}_{k}(SM)$, then $v_{l}=0$ for all $l$ with $|l|>k$. Thus the formula
above implies that $u_{n}=0$ for all $n$ with $|n|\geq k$.
A similar argument can be done for all other representations using the formulas in \cite[Section 4]{FF}.
In all the cases we see that $v\in F^{\infty}_{k}(SM)$ implies $u_{n}=0$ for all $n$ with $|n|\geq k$, and thus
$u\in  F_{k-1}^{\infty}(SM)$ as desired.

\end{proof}

Of course, for $k=1$ the proposition is saying exactly the same as Theorem B for $n=1$.






\section{Proof of Theorem B}

Given an arbitrary pair $(g,\Omega)$ on a closed manifold $M$, 
formula (\ref{pestov-integral-final}) takes the simpler form:

\begin{multline}\label{pestov-integral-finalr}
\int_{SM}\big\{|\mathbf X (\nabla^{\cdot} u)|^2
-\langle \mathbf R_y(\nabla^{\cdot} u),\nabla^{\cdot} u\rangle
-\langle\nabla^{\cdot } (\mathbf X u),Y(\nabla^{\cdot} u)\rangle
-2\langle Y(y),\nabla^{\cdot}u\rangle^2
\\+\langle \nabla^{:} u,Y(\nabla^{\cdot}u)\rangle
+\langle\nabla_{(\nabla^{\cdot}u)}Y(y),\nabla^{\cdot}u\rangle\big\}\,d\mu
=\int_{SM}\big\{|\nabla^{\cdot} (\mathbf X  u)|^2
-n(\mathbf X u)^2\big\}\,d\mu,
\end{multline}
where now all the derivatives that appear are obtained using the
Levi-Civita connection of the metric.

Suppose now that $\G(u)=h\circ\pi+\theta$ and extend $u$
to a positively homogeneous function of degree zero on
$TM\setminus\{0\}$ (still denoted by $u$).  Then $\mathbf X(u)=|y|\,
h\circ\pi+\theta$ and Lemma \ref{int-nabla} shows 
as in the proof of Theorem A that the right hand side of
(\ref{pestov-integral-finalr}) is non-positive and thus

\begin{multline}\label{lessz0}
\int_{SM}\big\{|\mathbf X (\nabla^{\cdot} u)|^2
-\langle \mathbf R_y(\nabla^{\cdot} u),\nabla^{\cdot} u\rangle
-\langle\nabla^{\cdot } (\mathbf X u),Y(\nabla^{\cdot} u)\rangle
-2\langle Y(y),\nabla^{\cdot}u\rangle^2
\\+\langle \nabla^{:} u,Y(\nabla^{\cdot}u)\rangle
+\langle\nabla_{(\nabla^{\cdot}u)}Y(y),\nabla^{\cdot}u\rangle\big\}\,d\mu
\leq 0.
\end{multline}

We also know that (\cite[Proof of Lemma 4.7]{DP2}:

\[\mathbf X(\nabla^{\cdot} u)=\nabla^{\cdot}(\mathbf X u)-\nabla^{:}u-\langle Y(y),\nabla^{\cdot} u\rangle y.\]
Therefore
\[-\langle\nabla^{\cdot } (\mathbf X u),Y(\nabla^{\cdot} u)\rangle
+\langle \nabla^{:} u,Y(\nabla^{\cdot}u)\rangle=-\langle \mathbf X(\nabla^{\cdot} u),Y(\nabla^{\cdot} u)\rangle
+\langle Y(y),\nabla^{\cdot} u\rangle^{2}.\]
Hence (\ref{lessz0}) gives:

\begin{multline}\label{lessz1}
\int_{SM}\big\{|\mathbf X (\nabla^{\cdot} u)|^2
-\langle \mathbf R_y(\nabla^{\cdot} u),\nabla^{\cdot} u\rangle
-\langle \mathbf X(\nabla^{\cdot } u),Y(\nabla^{\cdot} u)\rangle\\
-\langle Y(y),\nabla^{\cdot}u\rangle^2
+\langle\nabla_{(\nabla^{\cdot}u)}Y(y),\nabla^{\cdot}u\rangle\big\}\,d\mu
\leq 0.
\end{multline}

Let $M$ be as is the Theorem. Then $Y=\mathbb J$, $\nabla\mathbb J=0$, and the curvature tensor
is given by (see \cite[p. 166]{KN}):
\[-4\,R(X,Y,Z,W)=\langle X,Z\rangle\langle Y,W\rangle-\langle X,W\rangle
\langle Y,Z\rangle+\langle X,\mathbb J Z\rangle\langle Y,\mathbb J W\rangle\]
\[-\langle X,\mathbb J W\rangle\langle Y,\mathbb J Z\rangle+
2\langle X,\mathbb J Y\rangle\langle Z,\mathbb J W\rangle.\]
Let $Z:=\nabla^{\cdot} u$. Then $\langle Z,y\rangle=0$ and
\[R(y,Z,y,Z)=-\frac{1}{4}\big(|Z|^2+3\langle Z,\mathbb J y\rangle^{2}\big).\]

Using this information in (\ref{lessz1}) we obtain:

\begin{equation}\label{lessz2}
\int_{SM}\big\{|\mathbf X (Z)|^2
+\frac{1}{4}\big(|Z|^2-\langle Z,\mathbb J y\rangle^{2}\big)
+\langle \mathbb J(\mathbf X(Z)),Z\rangle\big\}\,d\mu
\leq 0.
\end{equation}

Write $Z=a\mathbb J y+Z_{0}$, where $\langle Z_{0},\mathbb J\rangle=0$.
Then
\[\mathbf X(Z)=\G(Z)=\G(a)\mathbb J y-ay+\G(Z_{0}),\]
\[\mathbb J(\mathbf X(Z))=-\G(a)y-a\mathbb J y+\mathbb J\G(Z).\]
Since 
\[\langle Z_{0},y\rangle=\langle Z_{0},\mathbb J y\rangle=0\]
we have
\[\langle \G(Z_{0}),y\rangle=\langle \G(Z_{0}),\mathbb J y\rangle=0.\]
Therefore
\[|\mathbf X(Z)|^2=[\G(a)]^2+a^2+|\G(Z_{0})|^2,\]
\[\langle \mathbb J(\mathbf X(Z)),Z\rangle=-a^2+\langle \mathbb J\G(Z_{0}),Z_{0}\rangle.\]
Hence (\ref{lessz2}) can be rewritten as follows
\[\int_{SM}\big\{[\G(a)]^2+|\G(Z_{0})|^2+\frac{1}{4}|Z_{0}|^{2}
+\langle \mathbb J\G(Z_{0}),Z_{0}\rangle\big\}\,d\mu\leq 0.\]
Equivalently
\begin{equation}\label{lessz3}
\int_{SM}\big\{[\G(a)]^2+|Z_{0}/2+\mathbb J\G(Z_{0})|^2\big\}\,d\mu\leq 0.
\end{equation}
This inequality can hold only if
\[\G(a)=0\;\;\;\mbox{\rm and}\;\;\;Z_{0}/2=-\mathbb J\G(Z_{0}).\]
The last equation implies
\[\G(|Z_{0}|^2)=2\langle \G(Z_{0}),Z_{0}\rangle=\langle \mathbb J(Z_{0}/2),Z_{0}\rangle=0.\]
Therefore $a$ and $|Z_{0}|^2$ are first integrals of $\G$. Since the magnetic flow
$\phi^1$ is ergodic with respect to the Liouville measure (cf. Lemma \ref{ergo})
we conclude that $|Z|$ is constant everywhere.
Let us show that this implies that $Z=\nabla^{\cdot} u =0$ everywhere.

Recall that $\nabla^{\cdot}u:=(u^{\cdot i})$ where $u^{\cdot i}:=g^{ij}u_{\cdot j}$
and $u_{\cdot j}:=\frac{\partial u}{\partial y^{j}}$.
Fix $x_{0}\in M$ and consider the restriction $\tilde{u}$ of $u$ to $S_{x}M$. Since $S_{x}M$ is compact
there is $y_{0}\in S_{x}M$ for which $d_{y_{0}}\tilde{u}=0$. Since $u$ is homogeneous of degree zero
we must have $\nabla^{\cdot}u(x_{0},y_{0})=0$. Thus $\nabla^{\cdot}u=0$ everywhere in $SM$.

\qed

\section{Proof of Theorem C}

The first thing to observe is that the existence of Green subbundles implies that $\tau(\Sigma)=M$
(see \cite[Corollary 1.13]{CI}). Since $H$ is convex and superlinear, for each $x\in M$, there
exists a unique $\beta_x\in T^*_{x}M$ such that $p\mapsto H(x,p)$ achives it unique minimum
at $\beta_x$.

The map $x\mapsto\beta_{x}$ can be seen as a smooth 1-form and $\beta_x$ belongs to the interior
of the region bounded by $\Sigma\cap T_{x}^*M$ for all $x\in M$.

Consider the map $B:T^*M\to T^*M$ given by $B(x,p)=(x,p-\beta_{x})$. It is easy to check that
$B^*(\la)=\la-\tau^*\beta$ and that $B^*(\tau^*\Omega)=\tau^*\Omega$.
Hence if we let $\widetilde{\Omega}:=\Omega+d\beta$, $B$ is
a symplectomorphism between $(T^*M,-d\la+\tau^*\Omega)$ and $(T^*M,-d\la+\tau^*\widetilde{\Omega})$.
The Hamiltonian $H_{\beta}(x,p):=H(x,p+\beta_{x})=H\circ B^{-1}$ achieves its minimum on every fibre
at the zero section and thus, without loss of generality, we may assume that $\Sigma$ contains the zero section of
$T^*M$. But in that case we can define a Finsler metric $F$ on $M$
using homogeneity and declaring that $\Sigma$ corresponds to the unit cosphere bundle of $F$.

Let $\phi^F$ be the magnetic flow of $(F,\Omega)$ acting on $\Sigma$ with infinitesimal generator
$X_{F}$. By definition of $F$, there exists a positive function $f\in C^{\infty}(\Sigma)$ such
that $X_{H}=fX_{F}$. Thus $X_{H}(u)=\tau^*\theta(X_{H})$ if and only if
$X_{F}(u)=\tau^*\theta(X_{F})$. On account of Theorem A, the proof of Theorem C will be complete
once we prove that the magnetic flow $\phi^F$ is also free of conjugate points on $\Sigma$. 

Since $X_{H}=fX_{F}$, there exists a smooth function $s_{t}(x,p)$ such that
\[\phi^F_{t}(x,p)=\phi_{s_{t}(x,p)}(x,p).\]
Differentiating with respect to $(x,p)$ we obtain
\[d\phi^F_{t}=d\phi_{s_{t}}+(X_{H}\circ\phi_{s_{t}})\,ds_{t}.\]
Let $E\subset T\Sigma$ be the stable Green subbundle for $\phi$ (recall that we are assuming
that $\phi$ has no conjugate points). Since $X_{H}\in E$, the formula above shows
that $E$ is also $d\phi^F$-invariant. Since $E$ is a Lagrangian subbundle which 
intersects the vertical subspace trivially, it follows from \cite[Proposition 1.15]{CI}
that $\phi^F$ is also free of conjugate points as desired.

\qed

\section{Appendix}

In this appendix we collect some facts about magnetic flows on compact quotients
of complex hyperbolic space.

\subsection{Complex hyperbolic space}

Our reference for this subsection is \cite{G}.

Let $\C^{n,1}$ be the $(n+1)$-dimensional complex vector space consisting of
$(n+1)$-tuples

\[ z = \left[ \begin{array}{c}
z' \\
z_{n+1}
\end{array} \right]\in \C^{n+1}\] 
with Hermitian pairing
\[\langle z,w\rangle=z_{1}\bar{w}_{1}+\dots+z_{n}\bar{w}_{n}-z_{n+1}\bar{w}_{n+1}.\]
We denote the group of unitary automorphisms of $\C^{n,1}$ by $U(n,1)$. For any
unit complex number $\rho$, scalar multiplication by $\rho$ lies in $U(n,1)$;
the corresponding subgroup is the center of $U(n,1)$.

Complex hyperbolic space $n$-space $\H^n$ is defined to be the subset of
${\mathbb P}(\C^{n,1})$ consisting of negative lines in $\C^{n,1}$.
A vector $z$ is said to be negative (resp. null, positive) if $\langle z,\rangle$ 
is negative (resp. null, positive).

The boundary $\partial \H^n$ of $\H^n$ is the set of null lines in $\C^{n,1}$.

Let $PU(n,1)$ be the image of $U(n,1)$ in $PGL(\C^{n,1})$. The group $PU(n,1)$
is the full group of biholomorphisms of $\H^n$.

Consider $\C^n$ with the standard positive definite Hermitian inner product
\[\langle\langle z,w\rangle\rangle=z_{1}\bar{w}_{1}+\dots+z_{n}\bar{w}_{n}.\]
Complex hypebolic space can be identified with the unit ball $\mathbb B^n$
in $\C^n$ by considering the restriction to $\mathbb B^n$ of the map
that takes $z'\in \C^n$ to the line determined by
\[\left[ \begin{array}{c}
z' \\
1
\end{array} \right]\in \C^{n,1}.\] 
The map also identifies $\partial {\mathbb B}^n=S^{2n-1}$ with $\partial \H^n$.

Complex hyperbolic space can be naturally endowed with a $PU(n,1)$-invariant
K\"ahler structure which we normalize so that it has holomorphic sectional curvature
equal to $-1$ (hence the sectional curvatures range in the interval $[-1,-1/4]$).

The group $PU(n,1)$ acts transitively on $\H^n$ and on the unit sphere bundle
$S\H^n$. Let $O$ be the point in $\H^n$ determined by the line
\[\left[ \begin{array}{c}
0 \\
1
\end{array} \right]\in \C^{n,1}\] 
which corresponds to the origin in $\mathbb B^n$. The stabilizer of $O$ is
given by $U(n)$ which sits in $PU(n,1)$ as follows
\[U(n)\mapsto U(n,1)\mapsto PU(n,1)\]
where the first map is
\[A\mapsto \left[ \begin{array}{cc}
A& 0 \\
0& 1
\end{array} \right]\]
and the second map is just projection. ($U(n)$ is the maximal compact subgroup
of $PU(n,1)$.) Thus $\H^n$ is the rank one symmetric space
\[PU(n,1)/U(n).\]

The Lie algebra of $PU(n,1)$ equals $su(n,1)$ and consists of matrices of the form
\[M(X,\xi):=\left[ \begin{array}{cc}
X& \xi \\
\xi^*& -\mbox{tr}(X)
\end{array} \right]\]
where $X\in u(n)$ satisfies $X^*=-X$ and $\xi\in \C^n$. (For any matrix $Y$ we denote
its conjugate transpose by $Y^*$.)
The embedding $u(n)\mapsto su(n,1)$ corresponding to $U(n)\mapsto PU(n,1)$ is given by
\[X\mapsto \left[ \begin{array}{cc}
X-\frac{1}{n+1}\mbox{tr}(X){\mathbb I}_{n}& 0 \\
0& -\frac{1}{n+1}\mbox{tr}(X)
\end{array} \right].\]
It is easily seen that $T_{O}\H^n$ is given by the matrices of the
form $M(0,\xi)$ and the metric of the K\"ahler structure is just
\[\langle M(0,\xi),M(0,\eta)\rangle=4\langle\langle \xi,\eta\rangle\rangle.\]
(The factor of $4$ makes the holomorphic sectional curvature equal to $-1$.)
One can also check that the almost complex structure
$\mathbb J_{O}:T_{O}\H^n\mapsto T_{O}\H^n$ is given by
\[\xi\mapsto i\xi.\]
The subgroup $U(n)$ acts on $T_{O}\H^n$ simply by
\[\xi\mapsto A\xi.\]
Hence if we set
\[\xi_{0} = \left[ \begin{array}{c}
0 \\
1/2
\end{array} \right]\in \C^n,\] 
then $M(0,\xi_{0})$ has norm one and the subgroup of $U(n)$ that
stabilizes $\xi_{0}$ is $U(n-1)$ embedded as
\[A\mapsto \left[ \begin{array}{cc}
A& 0 \\
0& 1
\end{array} \right].\]
Thus
\[S\H^n=PU(n,1)/U(n-1).\]

The embedding $u(n-1)\mapsto su(n,1)$ is given by
\[X\mapsto \left[ \begin{array}{ccc}
X-\frac{1}{n+1}\mbox{tr}(X){\mathbb I}_{n-1}& 0&0 \\
0& -\frac{1}{n+1}\mbox{tr}(X)&0\\
0&0&-\frac{1}{n+1}\mbox{tr}(X)
\end{array} \right].\]
Set
\[N(\eta,\xi,t):=\left[ \begin{array}{ccc}
{\mathbb O}_{n-1}& \eta&\xi' \\
-\eta^*& it&\xi_{n}\\
(\xi')^*&\xi_{n}& -it
\end{array} \right]\]
where $t\in\re$, $\eta\in \C^{n-1}$ and
\[\xi= \left[ \begin{array}{c}
\xi' \\
\xi_{n}
\end{array} \right]\in \C^n.\]
The tangent space $T_{(O,\xi_{0})}S\H^n$ is then given by the set of all matrices
$N(\eta,\xi,t)$ as above. 
If we let $\pi:S\H^n\to\H^n$ be the canonical projection, then
its differential $d\pi_{(O,\xi_{0})}S\H^n\to T_{O}\H^n$ is given by
\[N(\eta,\xi,t)\mapsto M(0,\xi).\]
From this we can easily see that the vertical subspace $\mathcal V$ at
$(O,\xi_{0})$ is just the set of all matrices of the form
$N(\eta,0,t)$. The horizonal subspace $\mathcal H$ is given by
the matrices of the form $N(0,\xi,0)$.

\subsection{Geodesic and magnetic flows}

Let 
\[X:=N(0,\xi_{0},0)=\left[ \begin{array}{ccc}
{\mathbb O}_{n-1}&0&0 \\
0& 0&1/2\\
0&1/2& 0
\end{array} \right].\]
The one-parameter subgroup generated by $X$ is precisely the geodesic flow
of $\H^n$. That is, if we let $\phi_{t}:S\H^n=PU(n,1)/U(n-1)\to S\H^n$ be
\[\phi_{t}(gU(n-1))=ge^{tX}U(n-1)\]
then $\phi$ is the geodesic flow of $\H^n$. This follows from the fact that $X$ is horizontal and
its horizontal component is precisely $\xi_{0}$.
Note that
\[e^{tX}=\left[ \begin{array}{ccc}
{\mathbb I}_{n-1}&0&0 \\
0& \cosh (t/2)& \sinh(t/2)\\
0&\sinh(t/2)&\cosh(t/2)
\end{array} \right]\]
which commutes with $U(n-1)$.

Next observe that the vertical vector, whose vertical component is
$i\xi_{0}=\mathbb J_{O}(\xi_{0})$ is given by the matrix
\[V:=N(0,0,1/2)=\left[ \begin{array}{ccc}
{\mathbb O}_{n-1}&0&0 \\
0& i/2&0\\
0&0& -i/2
\end{array} \right].\]

It follows that
\[X_{\la}:=X+\la V=\left[ \begin{array}{ccc}
{\mathbb O}_{n-1}&0&0 \\
0& \la i/2&1/2\\
0&1/2& -\la i/2
\end{array} \right]\]
generates the magnetic flow $\phi^{\la}$ of the pair $(g,\la\Omega)$ where
$g$ is the Riemannian metric on $\H^n$ and $\Omega$ is the
K\"ahler form.
Note that $X_1$ gives rise to a {\it unipotent} flow.

\subsection{Smooth compact quotients}

A classical result of A. Borel \cite{Bo} asserts that there are always
cocompact lattices $\Gamma\subset PU(n,1)$ such that
$M:=\Gamma\setminus  \H^n$ is a smooth compact manifold (a locally symmetric space).
The magnetic flows $\phi^\la$ descend to
\[SM=\Gamma\backslash PU(n,1)/U(n-1)\]
and we still denote them by $\phi^\la$.

The following lemma should be well known to experts. We include its proof for completeness.

\begin{Lemma} For $|\la |<1$, the flow $\phi^\la$ is, up to a constant time change,
smoothly conjugate to $\phi^0$ with entropy $\sqrt{1-\la^2}\,h(\phi^0)$. The flow
$\phi^1$ has zero entropy and is ergodic {\rm(}in fact mixing{\rm)} with respect to the Liouville measure.
\label{ergo}
\end{Lemma}

\begin{proof} If $|\la |<1$ one can check that there exists $c\in PU(n,1)$ such that
\[c^{-1} X_{\la} c= \sqrt{1-\la^2}\, X_{0}.\]
Moreover, we can choose $c$ of the form
\[c=\left[ \begin{array}{ccc}
{\mathbb I}_{n-1}&0&0 \\
0& *& *\\
0& *& *
\end{array} \right]\]
which commutes with $U(n-1)$. Hence the map $g\mapsto gc$ descends to $SM$ to give a conjugacy
between $\phi^\la$ and a constant time change of $\phi^0$ with scaling factor $\sqrt{1-\la^2}$.

The fact that $\phi^1$ has zero entropy follows easily from $\phi^0_{t}\circ \phi^{h}_{s}=
\phi^{h}_{se^{-t}}\circ\phi^{0}_{t}$ for all $s,t\in\re$, where $\phi^h$ is the flow generated
by
\[U:=\left[\begin{array}{ccc}
{\mathbb O}_{n-1}&0&0 \\
0& -i/2&-i/2\\
0&i/2& i/2
\end{array} \right].\]
But $\phi^1$ and $\phi^h$ are conjugated because there exists $d\in PU(n,1)$
such that $d^{-1} X_{1} d=U$. In fact
\[d=e^{\frac{\pi}{2}V}=\left[ \begin{array}{ccc}
{\mathbb I}_{n-1}&0&0 \\
0& e^{i\pi/4}& 0\\
0& 0& e^{-i\pi/4}
\end{array} \right].\]
Note that $U$ spans the center of the Heisenberg Lie algebra
giving the stable bundle of $\phi^0$. When $n=1$, $\phi^h$ is the classical horocycle flow.

The ergodicity of $\phi^1$ follows immediately from
Moore's ergodicity theorem \cite{moore}.

\end{proof}

\begin{Remark}{\rm Ratner's theorem \cite{R}
describes completely all the ergodic invariant probability measures of
$\phi^1$ (they are all algebraic).

The explicit form of $d$ tells us that
the conjugacy $f:SM\to SM$ between $\phi^1$ and $\phi^h$ is given by:
\[f(x,y)=(x,\mathbb J y)\]
where $\mathbb J$ is the almost complex structure.}
\end{Remark}

Using what we just proved it follows easily that $\phi^1$ has no conjugate points. Indeed, since
$\phi^\la$ is Anosov for $|\la|<1$, it does not have conjugate points. But the no conjugate points
condition is closed, thus $\phi^1$ also has no conjugate points. Of course, one can also give a direct proof
of this fact using the known expression of the Riemann curvature tensor solving explicitly the magnetic Jacobi
equation.

\end{document}